\documentclass[11pt,a4paper]{amsart}
\usepackage[utf8x]{inputenc}
\usepackage[T1]{fontenc}

\usepackage[pdftex,linkcolor=black,pdfborder={0 0 0}]{hyperref} % Format links for pdf
\usepackage{calc} % To reset the counter in the document after title page
\usepackage{enumerate} 
\usepackage{amssymb}
\usepackage{amsthm}
\usepackage{bbold}
\usepackage{commonfiles/mystyle}
\usepackage{commonfiles/commonshortcuts}

\usepackage[a4paper, lmargin=0.1666\paperwidth, rmargin=0.1666\paperwidth, tmargin=0.1111\paperheight, bmargin=0.1111\paperheight]{geometry} %margins

\usepackage[all]{nowidow} % Tries to remove widows
\usepackage[protrusion=true,expansion=true]{microtype} % Improves typography, load after fontpackage is selected

\newcommand{\Ffield}[2]{\mathbb{F}_{#1^{#2}}}
\newcommand{\Fq}{\mathbb{F}_{q}}
\newcommand{\Fp}{\mathbb{F}_{p}}
\newcommand{\Fl}{\mathbb{F}_{l}}

\hypersetup{ 	
pdfsubject = {},
pdftitle = {Restrictions on endomorphism rings of jacobians and their minimal fields of definition},
pdfauthor = {Pip Goodman}
}

\author{Pip Goodman}
\title[Restrictions on endomorphism rings of jacobians]{Restrictions on endomorphism rings of jacobians and their minimal fields of definition}

\email{p.a.goodman@bristol.ac.uk}
\begin{document} 

\maketitle

\begin{abstract}
    Zarhin has extensively studied restrictions placed on the endomorphism algebras of  Jacobians for which the Galois group associated to their 2-torsion is insoluble and ``large''. Here, we examine what happens when this Galois group merely contains an element of ``large'' prime order. As part of this we obtain a partial converse to a result by Guralnick and Kedlaya on the endomorphism field.
\end{abstract}

\section{Introduction}
\label{Intro}

Let $K$ be a number field, and $f \in K[x]$ a polynomial of degree $n$ with no repeated roots. Let $C_{f}$ be a smooth projective model of the smooth affine curve $$y^2 = f(x).$$ Denote the Jacobian of $C_{f}$ by $J_{f}$, and the splitting field of $f$ by $K(f)$.

Zarhin has extensively studied restrictions placed on the endomorphism ring of $J_f$, when the Galois group $\Gal(f)=\Gal(K(f)/K)$ is insoluble and large with respect to $g$, the dimension of $J_f$ \cite{Zar00} \cite{Zar_modreps} \cite{Zarhin_clifford}. 
As part of an algorithm to compute the endomorphism algebra of a genus 2 Jacobian, Lombardo proves that if $f$ is irreducible and has degree 5, then $J_f$ is absolutely simple \cite[Proposition 4.7]{Lom}.

Inspired by the results of Lombardo and Zarhin, we investigate restrictions placed on the endomorphism algebra of $J_f$ when $\Gal(f)$ merely contains an element of ``large'' prime order. Most of our results in fact hold for more general abelian varieties, but for the sake of exposition we state them in this section for odd degree hyperelliptic Jacobians.

We write ${\rm End }(J_f)$ for the ring of $\bar{K}$-endomorphisms of $J_f$ and ${\rm End }^0(J_f)= {\rm End }(J_f) \otimes \QQ$ for its endomorphism algebra. Let $\zeta_n$ denote a primitive $n$-th root of unity.

\begin{theorem}
\label{2_torsion_all_thms_put_together_HE}
Let $p=2g+1$ be a prime such that $2$ is a primitive root modulo $p$. Suppose $f \in K[x]$ is irreducible of degree $p$. Then one of the following holds:
\begin{enumerate}
    \item $E:= \End^0(J_f)$ is a number field, and both
    \begin{enumerate}[a)]
        \item $2$ is totally inert in $E/\QQ$,
        \item the order ${\rm End}(J_f)$ is $2$-maximal in $E$.
    \end{enumerate}
    \item  $J_f$ is isogenous over $\bar{K}$ to the self product of an absolutely simple abelian variety with complex multiplication by a proper subfield of $\QQ(\zeta_p)$.
\end{enumerate}
\end{theorem}

 We prove the above theorem in a more general context in \S2, see Theorem \ref{general_l_is_primitive_root}. A major step in the proof of Theorem \ref{general_l_is_primitive_root} is to obtain a partial converse to a theorem of Guralnick and Kedlaya on the possible degrees of the minimal extension $L/K$ over which all endomorphisms of $J_f$ are defined, see Theorem \ref{deg_min_fiel_def_divisible_by_large_prime} and the remark afterwards. 
 
 It should also be noted that if the Galois group of $f$ is known, we can often find restrictions on $[E:\QQ]$ using Remark \ref{remark_divisiblity_restrictions_on_[E:Q]}.
 
 When the dimension of $J_f$ is odd, we are able to relax the condition on the order of $2$ modulo $p$.

\begin{theorem}
\label{2_torsion_all_thms_put_together_HE_g_odd}
Let $g$ be odd, and $p=2g+1$ a prime such the index of $\langle 2 \rangle$ in $(\ZZ/p\ZZ)^*$ is at most 2. Suppose $f \in K[x]$ is irreducible of degree $p$. Then one of the following holds:
\begin{enumerate}
    \item $E := \End^0(J_f)$ is a number field, and provided ${\rm End}(J_f)$ is $2$-maximal, either
    \begin{enumerate}[a)]
        \item $2$ is totally inert in $E/\QQ$, or
        \item \label{not inert} there are two distinct primes above $2$ with equal inertia degree in $E$.
    \end{enumerate}
    \item  $J_f$ is isogenous over $\bar{K}$ to the self product of an absolutely simple abelian variety with complex multiplication by a proper subfield of $\QQ(\zeta_p)$.
\end{enumerate}
%Furthermore, if $\langle 2 \rangle = (\ZZ/p\ZZ)^*$, then case \ref{not inert}) does not occur.
\end{theorem}

Again, we prove a more general statement in \S2, see Theorem \ref{2_torsion_all_thms_put_together}. We also prove the following extension of Lombardo's result mentioned earlier. For the more general case of abelian surfaces, see Theorem \ref{ab_surface}. 

\begin{theorem}
\label{deg 5 thm}
Let $f(x) \in K[x]$ be an irreducible polynomial of degree 5. Then one of the following holds:
\begin{enumerate}
\item ${\rm End}(J_{f}) \cong\ZZ$.
\item ${\rm End}(J_{f}) \cong\ZZ[\frac{1+ r \sqrt{D}}{2}]$, where $D \equiv 5 \mod{8}$, $D > 0$ is square-free and $2 \nmid r$.
\item ${\rm End}(J_{f}) \cong R$, where $R$ is a 2-maximal order in a degree 4 CM field, which is totally inert at 2.
\end{enumerate}
In particular $J_{f}$ is absolutely simple and does not have quaternion multiplication over $\bar{K}$.
\end{theorem}

\begin{example}
\label{genus_2_examples}
Keeping the notation of Theorem \ref{deg 5 thm}, we have the following examples:

\begin{table}[ht]

\centering % used for centering table
\begin{tabular}{c c c } % centered columns 
\hline
 ${\rm Gal}(f)$ & ${\rm End}(J_f)$ & $f(x)$\\ [0.5ex] % inserts table
%heading
\hline % inserts single horizontal line
$D_5 $ & $\ZZ[\frac{1+\sqrt{13}}{2}]$ & $x^5-19x^4+107x^3+95x^2+88x-16$ \\ 
$F_5$ & $\ZZ[\frac{1+\sqrt{5}}{2}]$ & $x^5+10x^3+20x+5$ \\
$F_5$ & $\ZZ[\zeta_5]$ & $x^5 - 2$ \\ % inserting body of the table
$F_5$ & $R$ & $ 52x^5 + 104x^4 + 104x^3 + 52x^2 + 12x + 1$  \\ [1ex] % [1ex] adds vertical space %I 
\hline %inserts single line
\end{tabular}
%\label{table:nonlin} % is used to refer this table in the text
\end{table}

\noindent where $R$ is the maximal order of the CM number field with defining polynomial $x^4 + x^3 + 2x^2 - 4x + 3$. We note that this field is cyclic, ramified only at 13, and 2 generates a maximal ideal.

\end{example}

Specifying the Galois group of $f$, we are able to refine Theorem \ref{deg 5 thm}. The proof of the following theorem follows from Corollary \ref{Field_def_AGL(1,q)}, Propositions \ref{field_def_CM_Dm} and \ref{field_def_C5}, and Corollary \ref{C5_with_CM}.

We denote by $L$ the minimal extension of $K$ over which all endomorphisms of $J_f$ are defined. Following \cite{GK17} we call $L$ the \emph{endomorphism field} of $J_f$. It is well-known that $L/K$ is finite and Galois  \cite[Theorem 4.1]{Silverberg}.

\begin{theorem}
\label{Fields_of_definition_thm}
Let $f(x) \in K[x]$ be a polynomial of degree 5.
\begin{enumerate}
    \item Suppose ${\rm Gal}(f) \cong F_5$. The following hold:
    \begin{enumerate}[(i)]
        \item If  ${\rm End}^0(J_{f})$ is a real quadratic field, then $L$ is the unique degree 2 extension of $K$ contained in $K(f)$.
        \item If  ${\rm End}^0(J_{f})\cong E$ is a degree 4 CM field, then it is cyclic and $L$ is the unique degree 4 extension of $K$ contained in $K(f)$. Moreover, $L=EK$.
    \end{enumerate}
    \item Suppose ${\rm Gal}(f) \cong D_5$. If ${\rm End}^0(J_{f})$ is a degree 4 CM field, then $L$ contains the unique degree 2 extension of $K$ contained in $K(f)$.
    \item Suppose ${\rm Gal}(f) \cong C_5$. Then $[L:K] \leq 2$. Furthermore, if ${\rm End}^0(J_{f})$ is a degree 4 CM field, then $K$ contains a  real quadratic field with discriminant congruent to 5 modulo 8.
\end{enumerate}
In particular, if ${\rm Gal}(f) \cong F_5$, then we have an explicit list, depending on $f$, for all possible $L$.

\end{theorem}

\begin{remark}
If $f$ has degree $6$, then part (3) of the above theorem remains valid.
\end{remark}

\begin{remark}
As the natural action of $F_5$ contains odd permutations, we have  in case $ (1 i)$ that $L = K(\sqrt{\Delta_f})$, where $\Delta_f $ denotes the discriminant of $f$.
\end{remark}

\begin{example}
Using the above remark along with magma we produced the table below of examples of case $ (1 i)$ in Theorem \ref{Fields_of_definition_thm}. Indeed, all polynomials $f(x)$ listed are of degree 5, have Galois group $F_5$ over $\QQ$ and satisfy $\End(J_f) \cong \ZZ[\frac{1+\sqrt{5}}{2}]$.

\begin{table}[ht]
%\caption{Nonlinear Model Results} % title of Table
\centering % used for centering table
\begin{tabular}{c c } % centered columns
%\hline\hline %inserts double horizontal lines
\hline
 Endomorphism field & $f(x)$ \\ [0.5ex] % inserts table
%heading
\hline % inserts single horizontal line
$\QQ(\sqrt{2})$ & $x^5 - 14x^3 - 84x^2 + 81x - 28$ \\ % inserting body of the table
$\QQ(\sqrt{5})$ & $x^5 -5x^3 + 5x - 4$\\
$\QQ(\sqrt{13})$ & $x^5 - 4x^3 - 46x^2 - 44x - 194$ \\
$\QQ(\sqrt{17})$ & $x^5 - 2x^4 + 67x^3 - 250x^2 + 488x - 235$ \\
$\QQ(\sqrt{29})$ & $x^5 - x^4 + 4x^3 + 106x^2 - 97x + 669$ \\
$\QQ(\sqrt{41})$ & $x^5 - x^4 + 29x^3 - 1025x^2 - 3154x - 17714$\\
$\QQ(\sqrt{53})$ & $x^5 - x^4 + 54x^3 - 167x^2 + 1018x - 69$ \\ [1ex] % [1ex] adds vertical space
\hline %inserts single line
\end{tabular}
%\label{table:nonlin} % is used to refer this table in the text
\end{table}

% examples in the table magma-ready:
% [
% x^5 - 14*x^3 - 84*x^2 + 81*x - 28,
% x^5 - 5*x^3 + 5*x - 4,
% x^5 - 4*x^3 - 46*x^2 - 44*x - 194,
% x^5 - 2*x^4 + 67*x^3 - 250*x^2 + 488*x - 235,
% x^5 - x^4 + 4*x^3 + 106*x^2 - 97*x + 669,
% x^5 - x^4 + 29*x^3 - 1025*x^2 - 3154*x - 17714,
% x^5 - x^4 + 54*x^3 - 167*x^2 + 1018*x - 69]

\end{example}

The examples provided in this paper come from either the database \cite{LMFDB_genus2_curves}, made easily available on \cite{lmfdb}, or were computed on magma \cite{magma} using code from Molin and Neurohr's article \cite{Molin_Neurohr_Period_matrices} and polynomials found in \cite{Jensen_Ledet_Yui} and \cite{LMFDB_deg_5_fields}, the second of which was also accessed through the LMFDB. 

\begin{acknowledgements}
The author would like to thank the organisers of the summer school `Explicit and computational approaches to Galois representations' held at the University of Luxembourg,  where this work began. The author would also like to thank Tim Dokchitser, Davide Lombardo and the anonymous referee for their useful comments and suggestions. This work was supported by the \hbox{EPSRC}.
\end{acknowledgements}

\section{Upper bounds on endomorphism rings}
In what follows $K$ is a number field, $A$ is an abelian variety over $K$ of dimension $g$ and $l \in \ZZ$ is a prime. If $B/K$ is also an abelian variety, then we say $A$ is isogenous to $B$ and write $A \thicksim B$ to mean $A$ is isogenous to $B$ over $\bar{K}$. We denote by ${\rm End }(A)$ the ring of $\bar{K}$-endomorphisms of $A$ and ${\rm End }^0(A)= {\rm End }(A) \otimes \QQ$ for its endomorphism algebra.

The absolute Galois group $G_K = {\rm Gal }(\bar{K}/K)$ acts on ${\rm End }(A)$ and hence on ${\rm End }^0(A)$. Its invariant subspace ${\rm End }(A)^{G_K} ={\rm End }_K(A) $ is the ring of  $K$-endomorphisms of $A$. Likewise we write ${\rm End }^0(A)^{G_K} ={\rm End }^0_K(A) $.

\begin{lemma}
 \label{alg_is_not_quaternion}
 Let $\OO$ be an order  in a division algebra $E = \OO \otimes \QQ$ of dimension $nm^2$ over $\QQ$, where $n=[Z(E):\QQ]$. Then any division algebra contained in $\OO \otimes \ZZ/l\ZZ$ has dimension less than or equal to $nm$ over $\Fl$.
 
 In particular, if $\OO \otimes \ZZ/l\ZZ$ is a division algebra then $m=1$, i.e., $E$ is a field.
 \end{lemma}
 
\begin{proof}
Let $R \subseteq \OO \otimes \ZZ / l\ZZ$ be a division algebra of maximal dimension. As $R$ is a finite division algebra, we have by Wedderburn's Theorem that $R\cong \Ffield{l}{d}$ for some $d$. Hence, the  elements of $R$ have minimal polynomials of degree at most $d$ over $\Fl$. In addition, there is $\bar{a} \in R$, whose minimal polynomial over $\Fl$ has degree $d$.

Let us choose $a \in \OO$ such that its image in $\OO \otimes \ZZ/l\ZZ$ coincides with $\bar{a}$. Then $Z(E)(a)$ is a field, so has dimension over $\QQ$ less than or equal to $nm$ \cite[Corollary 3.17, page 96]{noncomalg}. Hence the minimal polynomial of $a$ over $\QQ$ has degree at most $nm$. Since this must be also hold true for the image of $a$ in $\OO \otimes \ZZ / l\ZZ$, we obtain $d \leq nm$.
\end{proof}

We say that an order $\OO$ in a number field $E$ is \emph{$l$-maximal} if $l$ does not divide the index of $\OO$ in the ring of integers of $E$.

\begin{proposition}
\label{l-maximal}
 Suppose that $G_K = {\rm Gal} (\bar{K}/K)$ acts irreducibly on $A[l]$. Then the following hold:
\begin{enumerate}
    \item $E = {\rm End}^0_K(A)$ is a number field.
    \item The order ${\rm End}_K(A)$ is $l$-maximal in $E$.
    \item $l$ is totally inert in $E/\QQ$.
\end{enumerate}
In particular, $A$ is simple over $K$ and does not have quaternion multiplication over $K$.
\end{proposition}

\begin{proof}
As $G_K$ acts irreducibly on $A[l]$, Schur's Lemma gives us that the endomorphism algebra $\End_K(A[l]) = \End(A[l])^{G_K}$ is a division algebra. Furthermore, $\End_K(A[l])$ is finite, so by Wedderburn's Theorem it is a field. Hence the subalgebra $R:= {\rm End}_K(A) \otimes \ZZ/l\ZZ$ of $\End_K(A[l])$ is  also a field. It follows that $E = {\rm End}^0_K(A)$ contains no idempotents, and hence is a division algebra. By Lemma \ref{alg_is_not_quaternion}, $E$ is a number field.

As $R$ is a finite field of characteristic $l$, we have $R\cong \Ffield{l}{d}$ where $d = [E:\QQ]$. Hence $R$ contains elements of degree $d$. As $[E:\QQ] = d$ the preimages of these elements in $E$ must also have minimal polynomials of degree $d$. Let $\alpha \in {\rm End}_K(A)$ be one such element, and $\mu$ be its minimal polynomial. Note that $\alpha \otimes 1$ satisfies $\bar{\mu}$, the reduction of $\mu$ modulo $l$, and as $\alpha \otimes 1$ has degree $d$ over $\Fl$, we have that $\bar{\mu}$ is irreducible. Since $\Fl$ is a perfect field, $\bar{\mu}$ does not have repeated roots, hence $l \nmid {\rm disc}(\mu)$. Thus $\ZZ[\alpha]$ is an $l$-maximal order in $E$, and so is ${\rm End}_K(A)$.

Let $\lambda $ be a prime above $l$ in $E$. As we are considering endomorphisms defined over $K$, we have $A[\lambda]$ is a non-zero $G_K$ submodule of $A[l]$. But $G_K$ acts irreducibly on $A[l]$, so  $A[l]=A[\lambda]$. We conclude $l$ is totally inert in $E/\QQ$, see \cite[\S 5.1, Prop 2, pg 36]{Shimura_Book_Ab_vars_with_CM}.
\end{proof}

\begin{remark}
\label{remark_divisiblity_restrictions_on_[E:Q]}
In the proof of Proposition \ref{l-maximal}, we show $\FF = \End_K(A[l])$ is a field and $\End_K(A) \otimes \ZZ/l\ZZ$ is a subfield of $\FF$. As $\End_K(A)$ is a free $\ZZ$-module,  this shows $[E: \QQ]$ divides $[\FF:\Fl]$. It follows that if we  can determine $\End(A[l])^{G_K}=\End_K(A[l])$, then we can also limit the possible values of $[E:\QQ]$.
\end{remark}

We now briefly discuss the endomorphism field of an abelian variety, that is, the minimal extension of the base field over which all endomorphisms are defined.

The action of the absolute Galois group $G_K$ on $\End(A)$  induces a representation
$$ G_K \rightarrow {\rm Aut}\left(\End(A)\right)$$
with kernel ${\rm Gal}(\bar{K}/L)$, where $L$ is the endomorphism field of $A$.
In particular, we obtain an injection of abstract groups
$$ {\rm Gal}(L/K) \hookrightarrow {\rm Aut}\left(\End(A)\right).$$

Thus if ${\rm Gal}(L/K)$ is large with respect to $g$, the dimension of $A$, then one might expect the structure of $\End(A)$ to be constrained. The following theorem shows this is indeed the case. But first we need a couple of lemmas.

For a division ring $E$, we shall write $\GL_d(E)$ for the invertible elements in $M_d(E)$, and $\PGL_d(E)$ for the quotient of $\GL_d(E)$ by its centre.

\begin{lemma}
\label{PGL_to_GL}
Let $E$ be a division ring, and $F$ a maximal subfield of $E$. Write $m=[E:F]$. Suppose $\PGL_d(E)$ contains an element $\bar{M}$ of order $n$, with $n$ coprime to $md$. Then $\GL_d(E)$ contains an element of order $n$.
\end{lemma}

\begin{proof}
Let $M$ be a lift of $\bar{M}$ to $\GL_d(E)$. We have $M^n = kI_d$ for some $k \in Z(E)^*$. Taking the reduced norm gives ${\rm Nrd}(M)^n = k^{md}$. As $n$ and $md$ are coprime there exists integers $a,b$ such that $an+bmd = 1$. Let $c = k^a {\rm Nrd}(M)^b$, then $c^n = k$ and as ${\rm Nrd}(M)\in Z(E)^*$ we have $c \in Z(E)^*$. The element $\frac{1}{c}M \in \GL_d(E)$ then has order $n$. Indeed,  $(\frac{1}{c}M)^n = I_d$, and if $(\frac{1}{c}M)^{n'} = I_d$ for some $n' < n$, then $\bar{M}$ would not have order $n$. 
\end{proof}

\begin{lemma}
\label{no_element_order_p}
Let $p=2g+1$ be prime, $d$ a divisor of $2g$, and $E$ be a division algebra of dimension less than or equal to $ 2g/d$ over $\QQ$. Suppose $\Aut_{\QQ}(M_d(E))$ contains an element of order $p$. Then $E$ is isomorphic to a subfield of  $\QQ(\zeta_p)$ and has degree $2g/d$.% only need d divides g, but why not?
\end{lemma}

\begin{proof}
Let $\sigma \in \Aut_{\QQ}(M_d(E))$ be an element of order $p$.  This element acts trivially on $Z(E)$, the centre of $E$, as $\Aut(Z(E)/\QQ)$ has order less than $p$. Thus $\sigma$ belongs to $\Aut_{Z(E)}(M_d(E))$, which by the Skolem-Noether Theorem, is isomorphic to $ \PGL_d(E)$. Applying Lemma \ref{PGL_to_GL}, we obtain an element of order $p$ in $\GL_d(E)$. 

Let $F$ be a maximal subfield of $E$, then $F$ is a finite extension of $\QQ$ in some fixed algebraic closure $\bar{\QQ}$ of $\QQ$. We have $M_d(E)\otimes F \cong M_{md}(F)$, where $m=[E:F]$, see \cite[Corollary 3.17, page 96]{noncomalg}. Hence we have an element $\bar{\sigma}$ of order $p$ in $\GL_{md}(F)$. Now consider $\QQ(\zeta_p)$ in our fixed algebraic closure $\bar{\QQ}$. The Galois group $\Gal(\QQ(\zeta_p)/\QQ(\zeta_p)\cap F)$ cyclically permutes the eigenvalues of $\bar{\sigma}$, which gives us that $|\Gal(\QQ(\zeta_p)/\QQ(\zeta_p)\cap F)| \leq md$. The inequality
$$2g = [\QQ(\zeta_p):\QQ] \leq md [F\cap \QQ(\zeta_p):\QQ]$$
coupled with
$$m[F:\QQ] = [E:F][F:\QQ] = [E:\QQ] \leq 2g/d$$
 shows $F$ is isomorphic to a field of dimension $2g/md$ over $\QQ$ contained in $\QQ(\zeta_p)$. Furthermore, as $\QQ(\zeta_p)/\QQ$ is a cyclic extension, there is a unique such subfield. In particular all maximal subfields of $E$ are isomorphic. But finite dimensional non-commutative division algebras over $\QQ$ have non-isomorphic maximal subfields \cite[Theorem 5.4]{Schacher}, thus $E=F$ is a field. % Another way to see this last statement is to note that the Grunwald Wang theorem provides us with infinitely many non-isomorphic cyclic subfields which can be used in the proof of Theorem 18.6 of Pierce Associative Algebras
%Indeed, in the notation of the book, for fixed n & m in G-W thm, there are infinitely many K/F, for m \neq 1. Since given any extension $K/F$ with subscribed set of places S, there is some totally split prime v. Thus applying G-W to S \cup \{ (v,m) \}, we get a cyclic degree m extn K'/F which is not totally split at v and so non-isomorphic to F. We can then use either of K and K' in the proof of Theorem 18.6 on the same page. 
% If m=1, then any degree m extension works, by corollary b ii) of the same book.
\end{proof}

\begin{theorem}
\label{deg_min_fiel_def_divisible_by_large_prime}
 Suppose $p = 2g+1$ is a prime divisor of $[L:K]$. Then 
 $A$ is isogenous over $\bar{K}$ to the self product of an absolutely simple abelian variety with complex multiplication by a proper subfield of $\QQ(\zeta_p)$.
\end{theorem}

\begin{proof}

 Let $B$ be a simple factor of $A_{\bar{K}}$, so that $A_{\bar{K}} \thicksim B^d \times C$ over $\bar{K}$ and $C$ has no factor isogenous to $B$. In this way, $B^d$ corresponds to a primitive central idempotent $e_B$ of the algebra $\End^0(A)$ satisfying $e_B\End^0(A) = \End^0(B^d)$. In addition, there is some positive integer $N$ such that $(Ne_B)(A) = B^d$.
 
 The Galois group $\Gal(L/K)$ acts on the (finite) set of primitive central idempotents in $\End^0(A)$. By assumption the order of $\Gal(L/K)$ is divisible by $p$, and so contains an element $\sigma$ of order $p$. Either $\sigma$ fixes $e_B$ or the orbit of $e_B$ under the action of $\sigma$ has size $p$. But $A$ has dimension less than $p$, so the latter is not possible, hence  $\sigma$ fixes $e_B$. It follows that $\sigma$ fixes $\End^0(B^d)$. Indeed, for $\varphi \in \End^0(B^d)$, we have $\varphi^\sigma = (e_B \varphi)^\sigma = e_B \varphi^\sigma \in \End^0(B^d)$. This also shows $B^d$ is defined over $L^{\langle \sigma \rangle}$, for $ \sigma(B^d) = \sigma(N e_B A) = N e_B A = B^d$.
 
As $\sigma$ acts non-trivially on $ \End(A)$, we may, after possibly exchanging $B$ for a different simple factor of $A_{\bar{K}}$, assume $\sigma$ acts non-trivially on $ \End(B^d)$. Thus  the cyclic group $\Gal(L/L^{\langle \sigma \rangle})$ of order $p$ acts faithfully on $ \End(B^d)$. Viewing $B^d$ as an abelian variety over $L^{\langle \sigma \rangle}$, we see the minimal extension over which all its endomorphisms are defined has degree divisible by $p$. Applying \cite[Theorem 4.1]{Silverberg}, we obtain the dimension of $B^d$ is at least $g$. Thus we find that $A_{\bar{K}} \thicksim B^d$. Since $\dim (A) = \dim (B^d) = d \dim(B)$, we have $\dim(B) = g/d$.  By \cite[Proposition 4.3]{Silverberg}, the dimension of $\End^0(A)$ over $\QQ$ is at least $p$, in particular $A$ is not absolutely simple, thus $d>1$.

Albert's Classification  \cite[Page 202]{Mumford_ab_vars_book} tells us that $E = \End^0(B)$ has dimension less than or equal to ${2g/d}$ over $\QQ$, so by Lemma \ref{no_element_order_p}, $E$ is isomorphic to a subfield of $\QQ(\zeta_p)$ of degree $2g/d$. Finally, we observe $B$ is a CM abelian variety as $E$ is a field and $[E:\QQ] = 2\dim(B) $. 
\end{proof}

\begin{remark}
The above theorem gives a partial converse to the main theorem of \cite{GK17}. Indeed, there the authors give an upper bound on the powers of primes dividing $[L:K]$ in terms of $g$, and show it is the best possible by exhibiting families of abelian varieties which achieve this bound.
\end{remark}

We shall use the following corollary in our proof of Theorem \ref{deg 5 thm}. We note that it may be deduced from results of \cite{FKRS12} (see also \cite[Definition 2.7]{Lom} for a summary), since they show the endomorphism field of an abelian surface is an extension of the ground field with degree dividing 48. Likewise, another proof of this fact is given in \cite{GK17}. But for the convenience of the reader we provide a direct proof using Theorem \ref{deg_min_fiel_def_divisible_by_large_prime}.

\begin{corollary}
\label{5_does_not_divide_deg_of_extension}
Suppose $g=2$, then $5$ does not divide $[L:K]$.
\end{corollary}

\begin{proof}
Suppose $5$ divides $[L:K]$. Then by Theorem \ref{deg_min_fiel_def_divisible_by_large_prime}, $A$ is isogenous to the square of an elliptic curve with CM by a proper subfield of $\QQ(\zeta_5)$. But all proper subfields of $\QQ(\zeta_5)$ are totally real, so we have a contradiction.
\end{proof}

Our next lemma provides us with a useful criterion for the action of $G_K$ on $A[l]$ to be irreducible.

 \begin{lemma}
 \label{faithful_module}
 Let $l$ and $p$ be primes and suppose the multiplicative order of $l$ modulo $p$ is $s>1$. Let  $V$ be a faithful representation of $\ZZ/p\ZZ$  over $\Fl$. Then $V$ is irreducible if and only if $V$ has dimension $s$.
 \end{lemma}

\begin{proof}
Let $0 \subsetneq V' \subseteq V$ be the largest submodule of $V$ not containing a copy of the trivial representation. The smallest extension of $\Fl$ which contains an element of multiplicative order $p$ is $\FF_{l^s}$. It follows that $\FF_{l^s}$ is the smallest extension of $\Fl$ over which all irreducible characteristic $l$ representations of $\ZZ/p\ZZ$ are linear. The $\FF_{l^s}[\ZZ/p\ZZ]$-module $V'\otimes \FF_{l^s}$, being semisimple thanks to Maschke's Theorem, is a sum of irreducible faithful linear $\FF_{l^s}$-representations of $\ZZ/p\ZZ$. As the character attached to $V'$ takes values in $\Fl$, the Galois group $\Gal(\FF_{l^s}/\Fl)$ permutes the irreducible constituents of $V'\otimes \FF_{l^s}$, so $V'$ has dimension $ms$ for some $m$. Furthermore, as Schur indices in positive characteristic are trivial \cite[Theorem 9.14]{Isaacs}, each orbit under the action of $\Gal(\FF_{l^s}/\Fl)$ is an irreducible $\Fl[\ZZ/p\ZZ]$-module. Thus $m$ is equal to the number of irreducible $\Fl[\ZZ/p\ZZ]$-modules in $V'$. It follows that $V$ is irreducible if and only if $V'=V$ and $m=1$, i.e., $V$ has dimension $s$.
\end{proof}

\begin{theorem}
\label{general_l_is_primitive_root}
Let $l$ and $p=2g+1$ be primes with $l$ a primitive root modulo $p$. Suppose $\Gal(K(A[l])/K)$ contains an element of order $p$. Then one of the following holds:
\begin{enumerate}
    \item $E := \End^0(A)$ is a number field, and both
    \begin{enumerate}[a)]
        \item $l$ is totally inert in $E/\QQ$,
        \item the order ${\rm End}(A)$ is $l$-maximal in $E$.
    \end{enumerate}
    \item $\Gal(L/K)$ contains an element of order $p$, and $A$ is isogenous over $\bar{K}$ to the self product of an absolutely simple abelian variety with complex multiplication by a proper subfield of $\QQ(\zeta_p)$.
\end{enumerate}
\end{theorem}
 
 \begin{proof}
 Suppose $p$ does not divide $[L:K]$. Then $\Gal(L(A[l])/L)$ contains an element of order $p$. The action of $\Gal(L(A[l])/L)$ on the $2g=p-1$ dimensional vector space $A[l]$ over $\Fl$ is faithful, and thus irreducible by Lemma \ref{faithful_module}. Applying Proposition \ref{l-maximal} we deduce the first case.
 
 Suppose now $p$ divides $[L:K]$. Then $\Gal(L/K)$ contains an element of order $p=2g+1$, so we can apply Theorem $\ref{deg_min_fiel_def_divisible_by_large_prime}$ to conclude.
 \end{proof}
 
 Theorem \ref{2_torsion_all_thms_put_together_HE} now follows by taking $l=2$ and observing that $\Gal(K(f)/K)= \Gal(K(J_f[2])/K)$, see \cite[Theorem 5.1]{Zar_modreps}.

\begin{theorem}
\label{ab_surface}
Suppose $g=2$, and $\Gal(K(A[2])/K)$ contains an element of order $5$. Then one of the following holds:
\begin{enumerate}
\item ${\rm End}(A) \cong\ZZ$.
\item ${\rm End}(A) \cong\ZZ[\frac{1+ r \sqrt{D}}{2}]$, where $D \equiv 5 \mod{8}$, $D > 0$ is square-free and $2 \nmid r$.
\item ${\rm End}(A) \cong R$, where $R$ is a 2-maximal order in a degree 4 CM field, which is totally inert at 2.
\end{enumerate}
In particular $A$ is absolutely simple and does not have quaternion multiplication over $\bar{K}$.
\end{theorem}

\begin{proof}
Observe 2 is a primitive root modulo 5, so we may apply Theorem \ref{general_l_is_primitive_root}. Furthermore, by Corollary \ref{5_does_not_divide_deg_of_extension}, we are in the first case of Theorem \ref{general_l_is_primitive_root}. Thus $E = {\rm End}^0(J_f)$ is a field, 2 is totally inert in the extension $E/\QQ$, and ${\rm End}(J_f)$ is a 2-maximal order in $E$.

As $E = {\rm End}^0(J_f)$ is a field, we have $[E:\QQ]$ divides ${ \rm 2dim}(J_f)=4$. If $[E:\QQ]=1$, then we are in case 1. If $[E:\QQ]=4$, then $E$ is a CM field.

If $[E:\QQ]=2$, then we may write $E=\QQ(\sqrt{D})$ for some square-free $D$. Since 2 is totally inert in the extension $E/\QQ$, we find that $D \equiv 5 \mod{8}$. The centre of the endomorphism algebra of an abelian surface cannot contain an imaginary quadratic field (see \S4 of \cite{Shimura63_albert_paper}), so we have $D>0$. Finally as $D \equiv 1 \mod{4}$, and ${\rm End}(J_f)$ is a 2-maximal order, we find ${\rm End}(J_f) \cong\ZZ[\frac{1+ r \sqrt{D}}{2}]$ with $2 \nmid r$.
\end{proof}

It is clear that Theorem \ref{deg 5 thm} follows from Theorem \ref{ab_surface}, since if $f \in K[x]$ is irreducible and has degree 5, then the Jacobian $J_f$ is an abelian surface and $\Gal(K(f)/K) = \Gal(K(J_f[2])/K)$ contains an element of order $5$.

We now start working towards the proof of Theorem \ref{2_torsion_all_thms_put_together_HE_g_odd}.

% Another reference for "centre of the endomorphism algebra of an abelian surface cannot contain an imaginary quadratic field" is Birkenhake & Lange, Complex abelian varieties - exercise 9.10 (4)

\begin{lemma}
\label{symplectic}
Let $l$ and $p = 2g+1$ be distinct primes such that $g$ is odd and the multiplicative order of $l$ modulo $p$ is $g$. Let $V$ be a symplectic space of dimension $2g$ over $\FF_l$, and $\rho : \ZZ/p\ZZ \rightarrow V$  a faithful representation  which preserves the symplectic pairing on $V$.

Then there exist irreducible non-isomorphic $\FF_l \langle \sigma \rangle$-submodules $U,W \subseteq V$ of dimension $g$ such that $V = U \oplus W$.
\end{lemma}

\begin{proof}
Let $\sigma$ denote a generator of $\ZZ/p\ZZ$. As $V$ is a faithful $\FF_l\langle \sigma \rangle$-module, and $\sigma$ has prime order $p \neq l$, there exists $U$ a faithful irreducible $\FF_{l}\langle \sigma \rangle$-submodule of $V$ which by Lemma \ref{faithful_module} has odd dimension $g$. By Maschke's Theorem we have a decomposition of $ \FF_l\langle \sigma \rangle$-modules $V = U \oplus W$.  As $U$ is irreducible and its dimension $g$ is odd, it follows from \cite[ Lemma 3.1.6]{Burness} that  $U$ is a totally isotropic subspace of $V$ with respect to the symplectic pairing. In order to preserve the pairing, we have that if $\sigma $ acts on $U$ as the matrix $M$, then $\sigma $ acts on $W$ as $(M^T)^{-1}$ (see \cite[Lemma 2.2.17]{Burness}). Thus $\sigma$ acts irreducibly on $U$ and $W$.

It rests to show $U$ and $W$ are not isomorphic as $\FF_l\langle \sigma \rangle$-modules. Indeed, suppose they were, so we have $\varphi:U \rightarrow W$ an isomorphism of $\FF_l\langle \sigma \rangle$-modules. Let $u \in U \otimes \bar{\FF}_l$ be an eigenvector of $M$ with eigenvalue $\mu \in \bar{\FF}_l$. Then by extending $\varphi$ linearly to an $\bar{\FF}_l$ homomorphism, we have 
\begin{equation*}
    \begin{split}
        \mu \varphi(u) & = \varphi(\mu u) \\
& =  \varphi(Mu) \\
&= \varphi(\sigma u ) \\
&= \sigma \varphi(u)\\
&= (M^T)^{-1} \varphi(u).\\
    \end{split}
\end{equation*}
But as $g$ and $|\sigma |$ are odd, $(M^T)^{-1}$ and $M$ have no common eigenvalues  \cite[Lemma 3.1.13]{Burness}. Thus we have a contradiction. Hence $U$ and $W$ are not isomorphic as $\FF_l\langle \sigma \rangle$-modules.
\end{proof}

The following lemma is Proposition 1.3 of \cite{Remond}.

\begin{lemma}
\label{field_def_of_homs}
Let $A$ and $B$ be abelian varieties defined over $K$, and $\varphi: A \rightarrow B$ an isogeny. Suppose $\End^0(A) = \End^0_K(A)$. Then $\varphi$ is defined over an extension of $K$ of (relative) degree dividing the order of the torsion multiplicative subgroup of the centre of $\End^0(A)$.
\end{lemma}

\begin{lemma}
\label{not_prod_2_nfs}
Suppose $g$ is odd and $\End^0(A) = \End^0_K(A)$. Furthermore, suppose $\Gal(K(A[l])/K)$ contains an element $\sigma$ of prime order $p=2g+1$ such that $A[l] = U \oplus W$ as $\Fl \langle \sigma \rangle $-modules, where both $U$ and $W$ are irreducible of dimension  $g$.

Then the endomorphism algebra of $A$ is not isomorphic to the product of two division algebras.
\end{lemma}

\begin{proof}
We first show any non-zero $G_K$ submodule of $V_l(A) = T_l(A) \otimes \QQ$ has dimension $g$ or $2g$. Let $M$ be the intersection of $T_l(A)$ with a non-zero $G_K$ submodule of $V_l(A)$. The image $M[l]$ of $M$ in $A[l]$ is a non-zero $G_K$ module, and thus equal to one of $U$, $W$, and $A[l]$. It follows that $M$ has dimension $g$ or $2g$.

Suppose $\End^0(A) \cong D_1 \times D_2$, for some division algebras $D_1$, $D_2$, so $A$ is isogenous to a product of non-isogenous, absolutely simple, abelian varieties $B \times C$. As  $\dim(B \times C) = g$ is odd, one of $B$ and $C$, has dimension less than or equal to $\frac{g-1}{2}$ and the other has dimension at least $\frac{g+1}{2}$. Thus $G_K$ must fix both $B$ and $C$, so $B$ and $C$ are defined over $K$, and in particular $V_l(B \times C)$ contains a non-zero $G_K$ submodule of dimension less than $g$. We shall show this is not possible.

Due to $B$ and $C$ being simple we have the bound $$\dim_\QQ(\End^0(B)), \dim_\QQ(\End^0(C)) \leq 2(g-1)$$ from which we deduce neither the centre of $D_1$ nor the centre of $D_2$ can contain a $p$-th root of unity. Thus by Lemma \ref{field_def_of_homs} we may, after replacing $K$ by a finite extension of degree not divisible by $p$, assume $A$ is isogenous to $B \times C$ over $K$. But this implies there is an isomorphism of $G_K$ modules $V_l(A) \cong V_l(B \times C)$, which is not possible since $V_l(A)$ has no non-zero $G_K$ submodule of dimension less than $g$.
\end{proof}

\begin{theorem}
\label{2_generates_index_2subgroup_for_PPAV}
Suppose $g$ is odd, $p=2g+1$ is prime and the order of $l$ modulo $p$ is $g$. Suppose $A/K$ is principally polarised and $p$ divides the order of $\Gal(K(A[l])/K)$. Then one of the following holds:
\begin{enumerate}
    \item $E = \End^0(A)$ is a number field, and when $\End(A)$ is $l$-maximal, either
    \begin{enumerate}[a)]
        \item $l$ is totally inert in $E/\QQ$, or
        \item \label{twoprimes} there are two distinct primes above $l$ with equal inertia degree in $E$.
    \end{enumerate}
    \item \label{big_case} $\Gal(L/K)$ contains an element of order $p$, and $A$ is isogenous over $\bar{K}$ to the self product of an absolutely simple abelian variety with complex multiplication by a proper subfield of $\QQ(\zeta_p)$.
\end{enumerate}
\end{theorem}

\begin{proof}

If the endomorphism field $L/K$ of $A$ has degree divisible by $p$, then by Theorem \ref{deg_min_fiel_def_divisible_by_large_prime} we are in case \ref{big_case}.

Suppose now the extension $L/K$ does not have degree divisible by $p$.
By enlarging $K$, we may assume $K$ contains $L$, $\Gal(K(A[l])/K)$ is cyclic of order $p$ generated by an element $\sigma$, and in particular $\End^0(A) = \End^0_K(A)$.

As $A$ is principally polarised, the Weil pairing is a symplectic pairing on $A[l]$. Furthermore, $A[l]$ is a faithful $\Fl\langle \sigma \rangle$-module, so we may apply Lemma \ref{symplectic}. Thus $A[l] = U \oplus W$, where $U$ and $W$ are irreducible non-isomorphic $\FF_l\langle \sigma \rangle$-modules of dimension $g$ over $\FF_l$. 
Hence we have the following decomposition $\End(A[l])^{\langle \sigma \rangle} = \End_{\langle \sigma \rangle}(A[l]) = \End_{\langle \sigma \rangle}(U)\oplus \End_{\langle \sigma \rangle}(W)$. By Schur's Lemma, we have that both $\End_{\langle \sigma \rangle}(U)$ and $\End_{\langle \sigma \rangle}(W)$ are division algebras. It follows that $\End(A[l])^{\langle \sigma \rangle}$ contains exactly two non-trivial idempotents.

Let us now consider $\End(A) \otimes \ZZ/l\ZZ $,  a subalgebra of $\End(A[l]) = \End(A[l])^{\langle \sigma \rangle}$. By the above $\End(A) \otimes \ZZ/l\ZZ $, and hence $\End(A)$, contains exactly zero or two non-trivial idempotents. If $\End(A)$ contains no non-trivial  idempotents, then $A$ is absolutely simple. Let us show that this is always the case. Suppose $\End(A)$ contains two non-trivial idempotents. Then $A \thicksim B \times C$ for some absolutely simple abelian varieties $B,C$. But this implies $\End(A)$ is isomorphic to the product of two division algebras, which is not possible by Lemma \ref{not_prod_2_nfs}.

So far we have proven that $E = \End^0(A)$ is a division algebra, and  $\End(A) \otimes \ZZ/l\ZZ $ contains a field of dimension at least $\frac{1}{2}[E:\QQ]$, so by Lemma \ref{alg_is_not_quaternion}, we have $[E:Z(E)] \leq 4$. Thus either $E$ is a number field, or $E$ is a quaternion algebra over $Z(E)$ \cite[Exrecise 31, page 106]{noncomalg}. But the later is not possible by Albert's Classification since $g$ is odd.

Let us now show the conditions on the primes above $l$ when $\End(A)$ is an $l$-maximal order. 
Let $\lambda$ be a prime above $l$ in $E$. Then $A[\lambda]$ is a non-trivial $\Fl \langle \sigma \rangle$-submodule of $A[l] = U \oplus V$. Thus $A[\lambda]$ must intersect at least one of $U$ and $V$ non-trivially, say $U$. By irreducibility of $U$ we have $U \subseteq A[\lambda]  $. If $A[\lambda]$ also intersects $V$ non-trivially, then for the same reason, we find $V \subseteq A[\lambda]$, and so $A[\lambda ]= A[l]$ which implies $l$ is totally inert \cite[\S 5.1, Prop 2, pg 36]{Shimura_Book_Ab_vars_with_CM}. Else, $A[\lambda] = U$ and $l$ is not totally inert. In this second case, let us first note that $l$ is not ramified since, $U$ and $V$, the Jordan-H\"older factors of $A[l]$ are not isomorphic as $\Fl\langle \sigma \rangle$-modules. Let $\lambda' \neq \lambda $ be another prime above $l$. Then $A[\lambda] \cap A[\lambda'] = 0$, and so by the above argument it follows that $A[\lambda'] = V$. Hence $A[l] = A[\lambda] \oplus A[\lambda'] $ and case \ref{twoprimes} is satisfied \cite[\S 5.1, Prop 2, pg 36]{Shimura_Book_Ab_vars_with_CM}.
\end{proof}

Theorem \ref{2_torsion_all_thms_put_together_HE_g_odd} is an immediate consequence of the following.

\begin{theorem}
\label{2_torsion_all_thms_put_together}
Suppose $g$ is odd and $p=2g+1$ is a prime such the index of $\langle 2 \rangle$ in $(\ZZ/p\ZZ)^*$ is at most 2. Suppose $A/K$ is principally polarised and there is an element of order $p$ in $\Gal(K(A[2])/K)$. Then one of the following holds:
\begin{enumerate}
    \item $E = \End^0(A)$ is a number field, and either
    \begin{enumerate}[a)]
        \item $2$ is totally inert in $E/\QQ$, or
        \item \label{not inert_general}there are two distinct primes above $2$ with equal inertia degree in $E$.
    \end{enumerate}
    \item  $\Gal(L/K)$ contains an element of order $p$, and $A$ is isogenous over $\bar{K}$ to the self product of an absolutely simple abelian variety with complex multiplication by a proper subfield of $\QQ(\zeta_p)$.
\end{enumerate}
Furthermore, if $\langle 2 \rangle = (\ZZ/p\ZZ)^*$, then case \ref{not inert_general}) does not occur.
\end{theorem}

\begin{proof}
Immediate from Theorems  \ref{general_l_is_primitive_root} and \ref{2_generates_index_2subgroup_for_PPAV}.
\end{proof}

\section{Bounds on the endomorphism field}
In order to prove Theorem \ref{Fields_of_definition_thm}, we shall need some results from the representation theory of finite groups.

Let $G$ be a finite group. Let $X$ be a finite set which $G$ acts on transitively. Suppose $l$ does not divide $|X|=n$.
Let $\Fl[X] $ denote the permutation module associated to this action over $\Fl$. The submodule given by $\langle \sum _{b \in X} b \rangle$ is easily seen to be isomorphic to the trivial module, and owing to the fact that $l \nmid n$, we have a direct sum decomposition of $\Fl[G]$-modules:
$$ \Fl[X] \cong  \Fl[X]^0 \oplus \Fl.$$

\begin{lemma}
\label{reptheory_bound_on_dim}
  Let $b \in X$ and suppose $G_b$  the stabiliser of $b$ has $s$ orbits on $X\setminus \{b\}$. Then $\dim_{\Fl}\End_{G}(\Fl[X]^0) \leq s$.
\end{lemma}

\begin{proof}
By Lemma 7.1 of \cite{Passman_Permutation_groups}, we have that $\dim_{\Fl}\End_{G}(\Fl[X]) = s+1$. Now due to the isomorphism of $\Fl[G]$-modules, $$ \Fl[X] \cong  \Fl[X]^0 \oplus \Fl$$  we have   $\dim_{\Fl}\End_{G}(\Fl[X]^0) \leq s$.
\end{proof}

In the following we shall denote the roots of a polynomial $f$ by $R_f$, and if $f$ is square-free, then we denote by $C_{f}$ a smooth projective model of the smooth affine curve $y^2 = f(x)$, and its Jacobian by $J_{f}$.

The following result may be deduced from \cite[Proposition 3]{Mori_II}, but we include a proof here for the convenience of the reader. 

\begin{proposition}
\label{bound_on_dim_over_K}
Let  $f \in K[x]$ be a polynomial of odd degree $n$ such that $G = {\rm Gal}(K(f)/K)$ acts transitively on $R_f$. Suppose $\alpha \in R_f$ and its stabiliser $G_\alpha$ has $s$ orbits on $R_f \setminus \{\alpha\}$. 

Then $\dim_ {\QQ} \left(\End_{K}^0(J_{f}) \right) \leq s$. In particular, if ${\rm Gal}(K(f)/K)$ acts doubly transitively on the roots of $f$, then $\End_{K}(J_{f}) \cong \ZZ$.
\end{proposition}

\begin{proof}
First, let us note $|R_f|=n$. As ${\rm Gal}(K(f)/K)$ acts  transitively on $R_f$ and $G_\alpha$ has $s$ orbits on $R_f\setminus \{\alpha\}$, with $n$ odd, we may apply Theorem 5.1 of \cite{Zar_modreps} and Lemma \ref{reptheory_bound_on_dim} to find that $\dim_{\FF_2} \left(\End_{G_K}(J_f[2])\right)\leq s$. Observe
$$\FF_2 \subseteq \End_{K}(J_{f})\otimes \ZZ/2\ZZ \subseteq \End_{G_K}(J[2]).$$
 It follows that $\dim_{\FF_2} \left(\End_{K}(J_{f})\otimes \ZZ/2\ZZ\right)\leq s$. Thus as $\End_{K}(J_{f})$ is a free $\ZZ$-module, we deduce $\dim_ {\QQ}  \left(\End_{K}^0(J_{f}) \right)\leq s$.
\end{proof}

Before proving our next result, we shall recall some standard facts on Frobenius groups and make a few remarks on the endomorphism field of an abelian variety.

A finite group $G$ is said to be a Frobenius group if $G$ contains a subgroup $H$ called the Frobenius complement such that
$$H \cap gHg^{-1} = \{1\} \text{ for all } g \in G\setminus H.$$
Frobenius proved that for such groups there exists a normal subgroup $N$ of $G$ such that $G \cong N \rtimes H$. We call $N$ the Frobenius kernel.
Left multiplication by elements of $G$ on $X$, the set of left cosets of $H$, gives rise to a transitive action and the above condition tells us the identity is the only element which fixes two or more elements of $X$. Hence, this action is at most doubly transitive and the stabiliser $G_b$ of a point $b \in X$ has no fixed points in $X\setminus \{b\}$. Applying Burnside's Lemma to $G_b$ and $X\setminus \{b\}$ (or $X$) shows this action is doubly transitive if and only if $|H| = |N| - 1 $, that is, $|G| = |N|^2 - |N|$. In this case, it can be shown that $N$ is a $p$-elementary group \cite[Propositon 8.4]{Passman_Permutation_groups}, and hence $N$ may be regarded as an $\Fp$ vector space and $G$ as a subgroup of the affine general linear group on this vector space.

Examples of Frobenius groups include  $F_q = \Fq \rtimes \Fq^\times \cong {\rm AGL}(1,q)$, $D_n$ for $n$ odd, and $C_{11}^2 \rtimes {\rm SL}_2(5)$, with Frobenius complements $\Fq^\times$, a subgroup of order 2  and ${\rm SL}_2(5)$ respectively.

\begin{lemma}
\label{Number_of_orbits_of_subgroup}
Let $G = N \rtimes H $ be a Frobenius group with complement $H$ and order $|N|^2 - |N|$. Suppose $U$ is a subgroup of $G$ with index $d$ dividing $|N|-1$.

Then $U$ acts transitively on $X$, the set of left cosets of $H$, and a point stabiliser $U_b$ has exactly $d$ orbits on $X\setminus \{ b\}$. 
\end{lemma}

\begin{proof}
As $UH = G$, we have $U$ acts transitively on $X$. For the second statement, observe that the non-identity elements of $U_b$ act fixed point freely on $X\setminus \{b \}$. Thus applying Burnside's Lemma, we find the number of orbits of $U_b$ on $X\setminus \{ b \}$ is equal to

$$ \frac{1}{|U_b|}\sum_{g \in U_b} |{\rm Fix}(g)| = \frac{d}{|N|-1} (|N|-1) = d.$$
\end{proof}

\begin{lemma}
\label{Clifford_gives_prime_splitting}
Let $G = N \rtimes H $ be a Frobenius group of order $|N|^2 - |N|$ with Frobenius kernel $N$ of odd order. Suppose $U$ is a subgroup of $G$ with index  dividing $|N|-1$.

Then the restriction of $\bar{\FF}_2[G/H]^0$ to $U$ is semisimple with non-isomorphic simple constituents.
\end{lemma}

\begin{proof}
As $G$ acts 2-transitively on the cosets $G/H$, the module $\CC[G/H]^0$ is irreducible by \cite[Cor. 5.17]{Isaacs} and has dimension $n-1$. It is thus irreducible and of 2-defect zero, so its associated module $\bar{\FF}_2[G/H]^0$ in characteristic 2 is also irreducible \cite[Thm 15.29]{Isaacs}. Furthermore, as the Frobenius kernel $N$ acts non-trivially on $\CC[G/H]^0$, we have by \cite[Thm 6.34]{Isaacs} that the restriction of $\CC[G/H]^0$ to $N$ is isomorphic to the direct sum of $n-1$ distinct non-trivial linear representations. Since $2$ does not divide the order of $N$, this also holds for the restriction of $\bar{\FF}_2[G/H]^0$ to $N$. Putting this together with the fact that $U$ is a normal subgroup, we see the restriction of $\bar{\FF}_2[G/H]^0$ to $U$ is semisimple with non-isomorphic irreducible consituents, by Clifford's Theorem \cite[Thm 6.5]{Isaacs}.
\end{proof}

Recall that we have an injection of abstract groups
$$ {\rm Gal}(L/K) \hookrightarrow {\rm Aut}\left(\End(J_f)\right)$$
where $L/K$ is the minimal extension over which all endomorphisms of $J_f$ are defined. In particular if $E = \End^0(J_f)$ is a number field, then $ {\rm Gal}(L/K) \hookrightarrow {\rm Aut}(E)$ and so $|{\rm Gal}(L/K)|$ divides $[E:\QQ]$. Recall also that if  $E = \End^0(J_f)$ is a number field then $[E:\QQ]$ divides $2\dim(J_f)$.

\begin{theorem}
\label{Field_def_Frob}
 Let $f \in K[x]$ be a polynomial of odd degree $n$ with Galois group isomorphic to a Frobenius group $G$ of order $n(n-1)$.

Suppose $E= {\rm End}^0(J_f)$ is a number field of dimension $s$ over $\QQ$. Then $E/\QQ$ is Galois with ${\rm Gal}(E/\QQ) $ isomorphic to a quotient of $H$ the Frobenius complement of $G$. 

Furthermore, $L/K$ is an extension of degree $s$ contained in $K(f)$, and as abstract groups, ${\rm Gal} (L/K) \cong  {\rm Gal}(E/\QQ)$. Moreover, if $s=n-1$, then $L=EK$.

Finally, if $\End(J_f)$ is $2$-maximal, then $E$ is unramified at 2.
\end{theorem}

\begin{proof}

By the above, we have that $|{\rm Gal}(L/K)|$ divides $ s$. Thus $[L:K]\leq s$ and $L\cap K(f) \subseteq F$ for some $F$, where $F/K$ is an extension of degree $s$ contained in $K(f)$.

Let $F' = L\cap K(f) $. By the above $F'/K$ an extension contained in $F$ of degree $d$ for some $d \in \NN$ dividing $s$. 
Since $E =\End^0(J_f)$ is a field, $J_f$ is absolutely simple and so $s$ divides $n-1 = 2 \dim(J_f)$. We may apply Lemma \ref{Number_of_orbits_of_subgroup}, to find the group ${\rm Gal}(K(f)/F') \cong {\rm Gal}(L(f)/L)$ acts transitively on $R_f$, and the stabiliser of a point $\alpha \in R_f$, has $d$ many orbits on $R_f\setminus \{ \alpha\}$. Thus, by Proposition \ref{bound_on_dim_over_K} we have that
$$s = \dim_{\QQ}\End^0(J_f) = \dim_{\QQ}\End^0_{L}(J_f) \leq d.$$
Hence $d=s$, and $L\cap K(f) = F$. But $[F:K]= s$ and $[L:K]\leq s$, so we find that $L=F$.

Since $|{\rm Aut}(E)| \leq s$ the injection of abstract groups ${\rm Gal}(L/K) \hookrightarrow {\rm Aut}(E)$ from above is an isomorphism. Whence we deduce that $|{\rm Aut}(E)|=s=[E:\QQ]$ and so $E/\QQ$ is Galois. We now conclude ${\rm Gal}(L/K) = {\rm Gal}(F/K)$ is isomorphic to a quotient of the Frobenius complement of $G$.

If $s=n-1$, then $J_f$ is a CM abelian variety. Furthermore, as $E/\QQ$ is Galois the reflex field $E^*$ is a subfield of $E$ \cite[Prop. 28, pg62]{Shimura_Book_Ab_vars_with_CM}. Since  $J_f$ is absolutely simple its CM type is primitive.  Proposition 30 on page 65 of \cite{Shimura_Book_Ab_vars_with_CM} applies to tell us the endomorphism field $L$ equals $E^*K$. In particular, $[E^*:\QQ] \geq s = [E:\QQ]$. It follows $E^* = E$ and $L= EK$.

It remains to show $E/\QQ$ is unramified at 2 when $\End(J_f)$ is 2-maximal. To this end, note that by Lemma \ref{Clifford_gives_prime_splitting} above and Theorem 5.1 of \cite{Zar_modreps} we have $J_f[2] \otimes \bar{\FF}_2$ is semisimple as a $\Gal(K(f)/L)$-module, with no repeated simple constituents. Thus $J_f[2]$ also has no repeated simple constituents as a $\Gal(K(f)/L)$-module. The result now follows since if $2$ ramified in $E$, then $J_f[2]$ would have repeated factors. 
\end{proof}

\begin{corollary}
\label{Field_def_AGL(1,q)}
Let $q$ be an odd prime power. Let $f \in K[x]$ be a polynomial of degree $q$ with Galois group $F_q = \Fq \rtimes \Fq^\times \cong {\rm AGL}(1,q)$.

Suppose $E= {\rm End}^0(J_f)$ is a number field of dimension $s$ over $\QQ$. Then $E/\QQ$ is cyclic Galois and $L/K$ is the unique extension of degree $s$ contained in $K(f)$.
If furthermore, $s=q-1$, then $L = EK$.

\end{corollary}

\begin{proof}
This follows immediately from Theorem \ref{Field_def_Frob} by noting that $F_q$ is a 2-transitive Frobenius group with cyclic Frobenius complement.

\end{proof}

\begin{example}
\label{higher_genus_examples}
In the table below we give examples of the above theorem. We denote by $\QQ(\zeta_p)^+$ the maximal totally real subfield of the $p$-th cyclotomic field $\QQ(\zeta_p)$.  We also write $L_f$ for the endomorphism field of $J_f$.

See also the table at the end of Section \ref{Intro} for more examples of Corollary \ref{Field_def_AGL(1,q)}.  

\begin{table}[ht]

\centering 
\small
\begin{tabular}{c c c c} 

\hline
${\rm Gal}(f)$ & $\End^0(J_f)$ & $L_f$ & $f(x)$ \\ [0.5ex]

\hline 
$F_7$ & $\QQ(\zeta_7)^+$ & $\QQ(\zeta_7)^+$ & $x^7 - 7x^5 + 14x^3 - 7x - 13$ \\ % inserting body of the table
$F_{11}$ & $\QQ(\zeta_{11})^+$ & $\QQ(\zeta_{11})^+$ & $x^{11} - 22x^9 + 176x^7 - 616x^5 + 880x^3 - 352x - 88$ \\
$F_{13}$ & $\QQ(\zeta_{13})^+$ & $\QQ(\zeta_{13})^+$ & $x^{13} + 13x^{11} + 65x^9 + 156x^7 + 182x^5 + 91x^3 + 13x - 2$ \\ 
$F_{17}$ & $\QQ(\zeta_{17})^+ $ & $\QQ(\zeta_{17})^+ $ & $x^{17} + 17x^{15} + 119x^{13} + 442x^{11} + 935x^9 +$ \\
& & & $1122x^7 + 714x^5 + 204x^3 + 17x - 1 $ \\
$F_{19}$ & $\QQ(\zeta_{19})^+ $ & $\QQ(\zeta_{19})^+ $ & $x^{19} + 19x^{17} + 152x^{15} + 665x^{13} + 1729x^{11} +$ \\
& & & $2717x^9 + 2508x^7 + 1254x^5 + 285x^3 + 19x - 1$ \\
[1ex] 
\hline 
\end{tabular}

\end{table}

\end{example}
\pagebreak

\begin{proposition}
\label{field_def_CM_Dm}
Let $f \in K[x]$ be an irreducible polynomial of odd degree $n$ with Galois group $D_n$.

Suppose $E= {\rm End}^0(J_f)$ is a number field of dimension $n-1$ over $\QQ$. Then $L \cap K(f)$ is the unique extension $F/K$ of degree $2$ contained in $K(f)$.
\end{proposition}

\begin{proof}
The proof is similar to that of Theorem \ref{Field_def_Frob}. Indeed, $[L:K]$ divides $[E:\QQ]=n-1$, and $[K(f):K]=2n$, so $[L \cap K(f):K]$ divides $\gcd(2n,n-1) = \gcd(2,n-1) =2$. As $D_n$ has a unique normal subgroup of index 2, we deduce the Galois extension $ L \cap K(f)/K$ is a subfield of the unique degree 2 extension $F/K$ contained in $K(f)$.

Let $F' = L\cap K(f) $. By the above either $F'= K$ or $F' = F$ thus $F'/K$ is an extension of degree $d$ diving $2$. 
As $D_n$ is a Frobenius group, ${\rm Gal}(K(f)/F') \cong {\rm Gal}(L(f)/L)$ acts transitively on $R_f$, and the stabiliser of a point $\alpha \in R_f$, is cyclic of order $2/d$, and has $\frac{n-1}{2/d} = \frac{d(n-1)}{2}$ many orbits on $R_f\setminus \{ \alpha\}$. Thus, by Proposition \ref{bound_on_dim_over_K} we have that
$$n-1 = \dim_{\QQ}\End^0(J_f) = \dim_{\QQ}\End^0_{L}(J_f) \leq \frac{d(n-1)}{2}.$$
Hence $d=2$, and $L\cap K(f) = F$.
\end{proof}
To finish the proof of Theorem \ref{Fields_of_definition_thm}, we require the following observation.

\begin{proposition}
\label{minimalfield_2group}
The Galois group of $L/K$, the minimal extension over which all endomorphisms of $A$ are defined, fits into the exact sequence
$$ 0 \rightarrow H \rightarrow {\rm Gal}(L/K) \rightarrow { \rm Gal}(K(A[2])/K),$$
where $H$ is a (possibly trivial) elementary abelian 2-group.
\end{proposition}

\begin{proof}
By Silverberg's result \cite[Theorem 4.1]{Silverberg}, we have that $L$ is contained in $K(A[4])$. 
Furthermore, ${\rm Gal}(K(A[2])/K)$ and ${\rm Gal}(K(A[4])/K)$ are subgroups of ${\rm GL}_{2g}(\ZZ/2\ZZ)$ and ${\rm GL}_{2g}(\ZZ/4\ZZ)$ respectively. Thus, as the following sequence is exact
$$ 0 \rightarrow I_{2g} + 2M_{2g}(\ZZ/4\ZZ) \rightarrow {\rm GL}_{2g}(\ZZ/4\ZZ) \rightarrow {\rm GL}_{2g}(\ZZ/2\ZZ) \rightarrow 0$$
and $I_{2g} + 2M_{2g}(\ZZ/4\ZZ) \cong M_{2g}(\ZZ/2\ZZ)$ is an elementary abelian 2-group (under addition) we are done.
\end{proof}

\begin{corollary}
Let $f \in K[x]$ be an irreducible polynomial of degree $n$ with Galois group $D_n$ and $n \equiv 3 \mod{4}$.

Suppose $E= {\rm End}^0(J_f)$ is a number field of dimension $n-1$ over $\QQ$. Then $L$ is the unique extension $F/K$ of degree $2$ contained in $K(f)$.
\end{corollary}

\begin{proof}
As $E$ is a number field, the discussion preceding Theorem \ref{Field_def_Frob} shows $[L:K]$ divides $n-1$. But the order of $ {\rm Gal}(K(f)/K) = {\rm Gal}(K(J_f[2])/K)$ is $2n$, so by Proposition \ref{minimalfield_2group} we have that ${\rm Gal}(L/K)$ is a 2-group.

Since $4$ does not divide $n-1$, it does not divide $[L:K]$ either. Thus $[L:K] \leq 2$. Hence, applying Proposition \ref{field_def_CM_Dm} we find $L=F$.  
\end{proof}

\begin{proposition}
\label{field_def_not_div_by_2}
Suppose $A/K$ is an abelian surface with CM by a degree 4 field and $2 \nmid [K(A[2]):K]$. Then $[L:K] \leq 2$ and $K$ contains a real quadratic field.
\end{proposition}

\begin{proof}
By assumption $E = {\rm End}^0(A)$ is a degree 4 CM field. Let us note that as $A$ is absolutely simple it has primitive CM type \cite[\S8.2, page 60]{Shimura_Book_Ab_vars_with_CM}. By \cite[\S8.4 Examples, pages 63-65]{Shimura_Book_Ab_vars_with_CM} there are only two such CM types for a genus 2 curve.  One of which is not Galois and the other is cyclic of degree 4. The minimal field of definition of the endomorphisms is $L = E^*K$ the compositum of $K$ and the reflex field of $E$, see \cite[Proposition 30, page 65]{Shimura_Book_Ab_vars_with_CM}.

In the non-Galois case, \cite[Proposition 5.17(5), page 131]{Shimura_Book_Intro_automorphic_functions} gives us that $K$ contains the real quadratic subfield of the reflex field, and thus also $[L:K] \leq 2$. In the other case, the reflex field $E^*$ is equal to $E=\End^0(A)$. As $[L:K]$ divides $4$ the image of ${\rm Gal}(L/K)$ in ${\rm Gal}(K(A[2])/K)$ is trivial, thus by Proposition \ref{minimalfield_2group} we have that ${\rm Gal}(L/K)$ has exponent dividing 2. Hence ${\rm Gal}(L/K)$ is isomorphic to either the trivial group, $C_2$, or $C_2 \times C_2$. If  ${\rm Gal}(L/K) \cong C_2 \times C_2$ held, then $E^*=E$ would have Galois group ${\rm Gal}(E^*/\QQ) \cong C_2 \times C_2$, but this implies the CM type of $A$ is not primitive \cite[\S8.4 Examples, 2A, page 64]{Shimura_Book_Ab_vars_with_CM}, contrary to what we have shown. Thus $[L:K] \leq 2$, and as $E$ contains a real quadratic subfield, we see that $K$ must too.
\end{proof}

\begin{proposition}
\label{field_def_C5}
Suppose $A/K$ is an abelian surface and ${\rm Gal}(K(A[2])/K) \cong C_5$. Then $[L:K] \leq 2$.
\end{proposition}

\begin{proof}
By Theorem \ref{ab_surface}, we have that ${\rm End}^0(A)$ is either isomorphic to $\QQ$, a real quadratic field, or a degree 4 CM field. In the first two cases $[L:K] \leq 2$ follows from the discussion before Theorem \ref{Field_def_Frob}, and in the last case it follows from Proposition \ref{field_def_not_div_by_2}.
\end{proof}

\begin{corollary}
\label{C5_with_CM}
 Let $A$ be as in Proposition \ref{field_def_C5}. If $A$ has CM, then $K$ contains a real quadratic field with discriminant congruent to  $5$ modulo $8$.
\end{corollary}

\begin{proof}
 By Proposition \ref{field_def_not_div_by_2}, $K$ contains a real quadratic field.
Let us now show the discriminant condition. Recall that discriminant of a real quadratic field is congruent to 5 modulo 8 if and only if 2 is inert. Theorem \ref{deg 5 thm} gives us that 2 is inert in $E/\QQ$, thus in the case that $E$ is Galois, the claim follows immediately by the above.

Let us suppose that $E$ is not Galois. Write $\QQ(\sqrt{d})$ for the real quadratic subfield of $E$, with $d \equiv 5 \mod{8}$ square-free (and hence equal to the discriminant of $\QQ(\sqrt{d})$). The reflex field of $E$ is also not a Galois extension of $\QQ$, and has a totally real subfield $\QQ(\sqrt{d'})$, where $d' \neq d$ is square-free \cite[8.4 Examples, 2C]{Shimura_Book_Ab_vars_with_CM}. By what we have already shown above, it suffices to prove $d'\equiv 5 \mod{8}$.

The number field $M$ obtained as the compositum of $E$ and its reflex $E^*$ is a Galois extension of $\QQ$, with Galois group ${\rm Gal}(M/\QQ) \cong D_4$ of order 8. We claim that $2$ splits in $M/\QQ$ and each prime above 2 has inertia degree 4. As 2 is inert in $E/\QQ$, it suffices to show $2$ splits in $M/\QQ$. We shall rule out the other possibilities. Let $D$ denote the decomposition group of a prime above $2$ in $M$. Note that as $D_4$ is not cyclic, $2$ is not totally inert in the extension $M/\QQ$. Next, suppose $2$ ramifies in $M/\QQ$, then by comparing orders, we see $D$ is isomorphic to $D_4$, but as the inertia degree is 4, it also has a cyclic quotient of order 4, which $D_4$ does not. Thus $2$ is not ramified in the extension $M/\QQ$. We deduce that $D$ is isomorphic to the unique degree 4 cyclic subgroup of $D_4$.

By examining the subgroup structure of $D_4$, we see the subfield fixed by $D$ is $\QQ(\sqrt{dd'})$ the quadratic subfield of $M$ not contained in $E$ or $E^*$. By definition of $D$ we have that 2 splits in $\QQ(\sqrt{dd'})$. Equivalently, $dd' \equiv 5d' \equiv 1 \mod{8}$, from which it follows that $d' \equiv 5 \mod{8}$.

% 2 does not divide d or d', so if dd' has a square factor it is odd and hence congruent to 1 mod 8
\end{proof}

\bibliographystyle{alpha}
\bibliography{References}

\end{document}